\documentclass[11pt]{article}
\usepackage{amsmath,amsthm,amscd,amsfonts,amssymb,enumerate}
\usepackage{graphicx}		
\usepackage{color}

\usepackage{tikz}
\usepackage{pgf,tikz}
\usetikzlibrary{arrows}
\tikzstyle{vertex}=[circle, draw, inner sep=0pt, minimum size=6pt]

\usepackage{amsmath,amsthm,amscd,amsfonts,amssymb,enumerate}
\usepackage{graphicx}		
\usepackage{color}
\newtheorem{prelem}{{\bf Theorem}}

\newtheorem{theorem}{Theorem}
\newtheorem{corollary}[theorem]{Corollary}
\newtheorem{lemma}[theorem]{Lemma}
\newtheorem{observation}[theorem]{Observation}
\newtheorem{proposition}[theorem]{Proposition}

\newtheorem{con}{Conjecture}
\theoremstyle{definition}
\newtheorem{definition}[theorem]{Definition}

\title{Chromatic and Clique number of Generalized Sierpi\'{n}ski Gasket Graph $S[G,t]$}

\date{}
\author
{Fatemeh Attarzadeh\thanks{\tt prs.attarzadeh@gmail.com} \\
{\it \small Department of Mathematics, University of Guilan, Rasht, Iran.}\\
	Ahmad Abbasi \thanks{\tt aabbasi@guilan.ac.ir} \\
	{\it \small Department of Mathematics, University of Guilan, Rasht, Iran.}\\
	{ Ali Behtoei\thanks{\tt Corresponding author, a.behtoei@sci.ikiu.ac.ir}}\\
	{\it \small Department of Mathematics, Faculty of Science,} \\ {\it \small Imam Khomeini International University,} \\
	{ \it \small Qazvin, Iran, PO Box: 34148 - 96818.}  
}
\begin{document}
	\maketitle
\begin{abstract}
In this paper we study the chromatic number and the clique number for the  generalized Sierpi\'{n}ski gasket $S[G,t]$ in which the base graph $G$ is an arbitrary simple graph.
For the clique number, we show that $\omega\big(S[G,t]\big)=\omega(G)$ for each graph $G$ and each positive integer $t$. 	
For the chromatic number, we prove that $\chi (S[G,2])=\chi (G)$ and $\chi (G) \leq\chi (S[G,t])\leq \chi (G)+1$ for each graph $G$ and each $t\geq 3$. Also, for each bipartite graph $G$ we show that $\chi (S[G,t])=\chi (G)$ and among some other results, we conjecture that $\chi (S[G,t])=\chi (G)$  for each graph $G$ and each $t\geq 3$.
\vspace{3mm}
\end{abstract}

\section{Introduction}

Sierpi\'nski type graphs appear naturally in diverse areas of mathematics and other scientific fields, see for example  \cite{D} and \cite{R}. 
Sierpi\'nski gasket graphs introduced in 1944 by Scorer, Grundy and Smith  \cite{F}, are one of the most important families of such graphs which are obtained after a finite number of iterations   and have a significant role in diverse areas.
Let $G=(V,E)$ be a simple  graph with  vertex set $V(G)$ and edge set $E(G)$. The set $N_G(u)$ denotes the (open) neighborhood of $u\in V(G)$, which means the set of all adjacent vertices  to $u$ in $G$ and the degree of vertex $u$ is $\deg_G(u)=|N_G(u)|$.
A clique $C$ in $G$ is a subset of vertices of $G$ such that every two distinct vertices in $C$ are adjacent and hence, the induced subgraph of $G$ on it is a complete graph. The clique number of $G$, $\omega(G)$, is the maximum size among all cliques of $G$.
A (proper vertex) $k$-coloring of $G$ is an assignment of $k$ colors to the vertices of $G$ in such a way that no pair of adjacent vertices receive the same color, and the chromatic number $\chi(G)$ is the minimum integer $k$ for which a $k$-coloring for $G$ exists.
Motivated by topological studies of the Lipscomb’s space, Klav$\check{z}$ar et al.  introduced the Sierpi\'{n}ski graph $S(K_n,t)$ in which the base graph is the complete graph $K_n$, see \cite{B} and \cite{T}. 
More generally, see \cite{P}, the $t$-th generalized Sierpi\'{n}ski of an arbitrary graph $G=(V,E)$, denoted by $S(G,t)$, is the graph with vertex set 
$V^t$ (the set of all words of length $t$ on the alphabet $V$) and two vertices
${\bf u}=u_1u_2 \dots u_t$ and ${\bf v}=v_1v_2 \dots v_t$ are adjacent in it if and only if there exist
$i \in \lbrace{1, \dots , t}\rbrace$ such that \\
$
(i) ~~u_j = v_j ~~if~~j < i,\\
(ii)~~ u_i \neq  v_i ~~and~~ u_iv_i \in E(G),\\
(iii) ~~u_j = v_i ~~and ~~v_j = u_i ~~if~~ j > i.\\
$
In this case, $\bf{uv}$ is considered as a linking edge appeared at step $i$.
For convenient,  we usually let $V=\{1,2,...,n\}$,  see Figure ${\ref{fig:pic7}}$.
\begin{figure}[ht] 
\centering
\includegraphics[scale=.7]{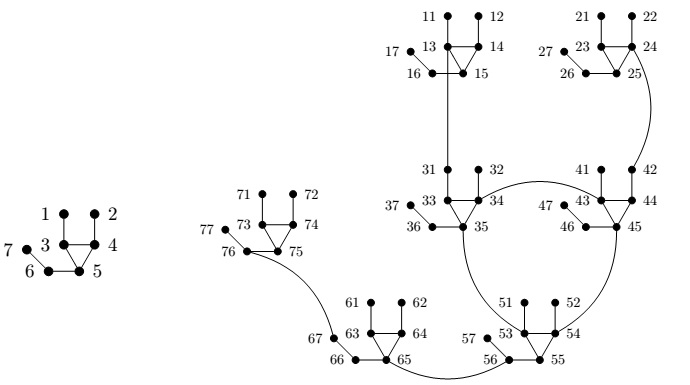}
\caption{\label{fig:pic7}
A graph $G$ and its generalized Sierpi\'{n}ski  $S(G,2)$.
}
\end{figure}

In general, $S(G,t)$ can be constructed recursively from the base graph $G$ with the following process: $S(G,1)=G$ and, for $t\geq 2$, we copy $n$ times $S(G,t-1)$ and add the letter $i$ at the beginning of each label of the vertices belonging to the copy of $S(G,t-1)$ corresponding to vertex $i$ (let $S_i(G,t)$ be the subgraph of $S(G,t)$ induced by these vertices which is also isomorphic to $S(G,t-1)\big)$. Then for each edge $ij\in E(G)$, we add an edge between two vertices $ijj...j$ and $jii...i$ (a linking edge).
Vertices of the form $ii...i$ (where $1\leq i \leq n$) are called extreme vertices.
Note that $S_i(G,t)$ is isomorphic to $S(G,t-1)$ and consequently, $S(G,t)$ contains $n^{t-1}$ copies of the graph $S(G,1)=G$. Also, if $ij\in E(G)$, then the vertex  $ijj...j$ in the copy $S_i(G,t)$ is adjacent to the vertex $jii...i$ in copy $S_j(G,t)$ and this is the unique edge between these two copies.
It is shown in \cite{Y} that the order of generalized Sierpi\'nski graph $S(G,t)$ is $n^t$ and its size is $|E(G)|~\!{n^t-1 \over n-1}~\!$.

Sierpi\'{n}ski graphs $S(K_n,t)$ are almost regular and possess many appealing properties, as for instance several coding and several metric properties \cite{Q}.
The graph $S(K_3,t)$ is isomorphic to the Tower of Hanoi game graph with $t$ disks, see \cite{B} and
polymer networks and WK-recursive networks can be modeled by generalized Sierpi\'{n}ski graphs, see \cite{PolymerNet}. 
In \cite{V} Parisse proved that $\chi(S(K_n,t))=n$.
Rodr\'iguez-Vel\'azquez et al. in \cite{Y}  obtained closed formulae for several graphical parameters  of the generalized Sierpi\'nski graph $S(G,t)$ (including chromatic, vertex cover, clique, independence and domination number) in terms of the parameters of the base graph $G$. In particular, they prove that $\omega(S(G,t))=\omega(G)$ and $\chi(S(G,t))=\chi(G)$. 
Also,  in \cite{FFF} the degree sequence of $S(G,t)$ is completely determined in terms of the degree sequence of $G$.   

As mentioned before, one of the most important families of fractal like graphs is the the family of Sierpi\'{n}ski gasket graphs (Sierpi\'{n}ski fractals) which appears frequently and has many applications in different areas.
The Sierpi\'{n}ski gasket graph $S_t$ is just a step from the Sierpi\'nski graph $S(K_3,t)$
and is constructed from the Sierpi\'{n}ski graph $S(K_3,t)$ by contracting every edge of $S(K_3,t)$ that lies in no triangle $K_3$
(i.e. by contracting all newly added edges, linking edges, during the iterations), see \cite{L}. 
The Sierpi\'{n}ski and the Sierpi\'{n}ski gasket  of $G=K_3$ at  step $t=3$ are shown in Figure  \ref{fig:pic2222}.
\begin{figure}[ht] 
\centering
\includegraphics[scale=.5]{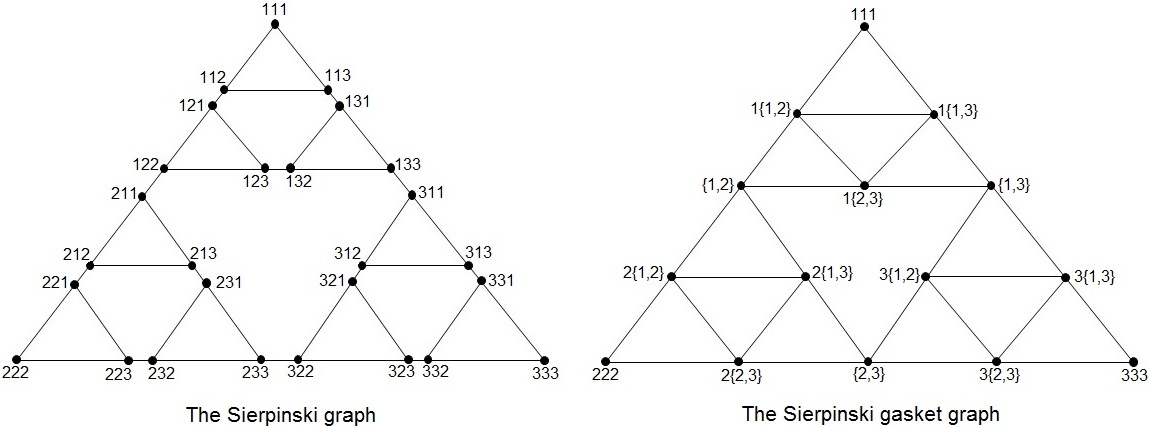}
\caption{\label{fig:pic2222}
The Sierpi\'{n}ski graph and the Sierpi\'{n}ski gasket graph at step $3$.
}
\end{figure}
The vertex-colorings, edge-colorings and total-colorings of the Sierpi\'nski gasket graphs and some other Sierpi\'nski like graphs are studied in \cite{ColorSerpLike}. 
In \cite{L} it is shown that the clique number and the chromatic number of $S_t$ is three, $S_t$ is uniquely $3$-colorable
and that the edge chromatic number (chromatic index) of $S_t$ is $4$ when $t\geq 2$.
In \cite{H} the same construction method is applied for any Sierpi\'{n}ski graph $S(K_n,t)$ by contracting edges that lie in no induced  complete subgraphs $K_n$ (i.e. by contracting all linking edges) and the resulting graph is called a generalized Sierpi\'{n}ski gasket graph which is denoted by $S[K_n,t]$.
In \cite{H} Jakovac studied several properties of graphs $S[K_n,t]$ including hamiltonicity and chromatic number.
He 
show that $\chi(S[K_n,t])=n$ and  that  $S[K_n,t]$ is  Hamiltonian. 
\section{Generalized Sierpi\'{n}ski Gasket ${\bf S[G,t]}$}

Since generalized Sierpi\'nski gasket graphs $S[K_n,t]$ are important and are naturally derived from Sierpi\'nski graphs $S(K_n,t)$ by contracting the linking edges, we  can apply the same for each (base) graph $G$ to construct the (more) generalized Sierpi\'{n}ski gasket graph $S[G,t]$, see also \cite{TOC}. 

\begin{definition}	\label{def1} 
Let $t$ be a positive integer and $G$ be a simple graph of order $n\geq 2$ with the vertex set $V(G)=\{1,2, \dots n\}$. 
The (more) generalized Sierpi\'{n}ski gasket graph $S[G,t]$ is obtained by contracting all of linking edges in the generalized Sierpi\'{n}ski $S(G,t)$ during the iteration processes.
\end{definition}

Similar to the structures of Sierpi\'{n}ski gaskets and generalized Sierpi\'{n}ski graph $S(G,t)$, the generalized Sierpi\'{n}ski gasket $S[G,t]$ can be constructed by using all the vertices of $G$  replaced by the copies of $S[G,t-1]$ and then contracting all linking edges between these copies.
Let $S_i[G,t]$ be the copy of $S[G,t-1]$ in $S[G,t]$ corresponding to the vertex $i \in V=\{1,2, \dots n\}$. 
Note that if $ij\in E(G)$, then the linking edge between two vertices $ijj \dots j$ and $jii \dots i$ in $S(G,t)$ is contracted in $S[G,t]$.
This (newly contracted) vertex will be denoted by $\{i,j\}_{t}$ in $S[G,t]$ $\big{(}$or, for convenient by $\{i,j\}$ when $t=2\big{)}$  and we say that the expanded forms of $\{i,j\}_{t}$ are $ijj \dots j$ and $jii \dots i~\!$.
Some times and for convenient, we identify a contracted vertex with its expanded form representations.
In fact $\{i,j\}_{t}$ is the unique common vertex of two copies $S_i[G,t]$ and $S_j[G,t]$,  see Figure \ref{fig:pic1616}.
Since $|V(S[G,t])|=n|V(S[G,t-1])|-|E(G)|$ and $|E(S[G,t])|=n|E(S[G,t-1])|$, the order of $S[G,t]$ is $$n^t-(n^{t-2}+\cdots +n+1)|E(G)|=n^t-|E(G)|~\!{n^{t-1}-1 \over n-1}$$ and its size is $|E(G)|~\! n^{t-1}$.
\begin{figure}[ht]
	\centering
	\includegraphics[scale=.8]{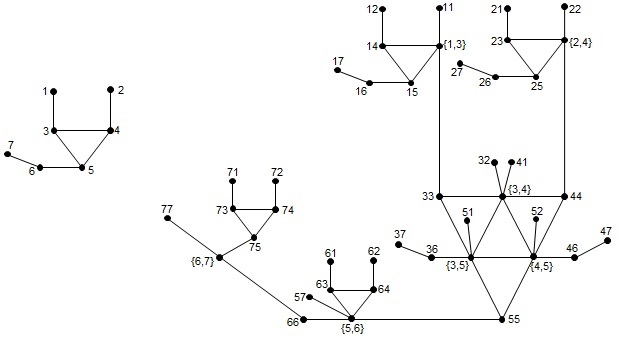}
	\caption{\label{fig:pic1616}
		A graph $G$ and its  generalized Sierpi\'{n}ski gasket $S[G,2]$.
	}
\end{figure}

Also, when $ij\in E(G)$, two vertices 	${\bf u}= u_1u_2 \dots u_r ijj \dots j$ and $ {\bf v}= u_1u_2 \dots u_r jii \dots i$ (in which $0 \leq r \leq t - 2$ and $u_l\in V$ for each $1\leq l\leq r$) are adjacent
in $S(G,t)$ and the linking edge  $\bf{uv}$ $\big{(}$which is actually produced at step $t-r$ in the iteration process of $S(G,t) \big{)}$ is contracted in $S[G,t]$ and this (contracted) vertex  may be denoted by $u_1u_2 \dots u_r \lbrace{i,j}\rbrace_{t-r}$ 
$\big{(}$or, for convenient by $u_1u_2 \dots u_r\{i,j\}$ when $t-r=2\big{)}$
whose expanded forms are $u_1u_2 \dots u_r ijj \dots j$ and $u_1u_2 \dots u_r jii \dots i~\!$.
 Klav$\check{z}$ar et al. in \cite{AAA} proved that when $G$ is a graph with $cc(G)$ connected components, then the generalized Sierpi\'{n}ski graph $S(G,t)$ has  	${1 \over n-1}\big{(}n^{t}(cc(G)-1)+n-cc(G)\big{)}$ 	connected components.
Since edge contraction do not change the number of connected components, the following result directly follows, and by considering the iteration process of producing $S[G,t]$, Observation \ref{ob2} will be obviously obtained.
\begin{proposition} \label{ob1}
	If $G$ is a graph with $cc(G)$  connected components, then the generalized Sierpi\'{n}ski gasket  $S[G,t]$ have 
	${1 \over n-1} 	\big{(}n^{t}(cc(G)-1)+n-cc(G)\big{)}$
	connected components.
\end{proposition}

\begin{observation}\label{ob2}
	If $G$ is an acyclic graph, then for each $t\geq 1$, $S[G,t]$ is acyclic. Also If $G$ contains $l$ cycles of order $k$, then $S[G,t]$ contains at least $~l ~\! n ^{t-1}$ cycles of order $k$.
\end{observation}

It is well known that the problem of finding a maximum clique is NP-complete, \cite{Garey}. We want to show that the clique number of the generalized Sierpi\'{n}ski gasket graph $S[G,t]$ is equal to the clique number of its base graph $G$. 
For this purpose, we need to determine the neighborhood of each vertex in $S[G,t]$. First, consider the structure of $S[G,2]$ and assumme that $i$ and $j$ are two non-adjacent vertices in $G$. 
Then the (non-contracted) vertex $ij$ in $S[G,2]$  is adjacent just to a vertex in the form $il$ for which $j$ and $l$ are adjacent in $G$. 
Since $i$ and $l$ may be adjacent or non-adjacent in $G$, we deduce that
$$N_{S[G,2]} (ij)=\big{\{} il: ~jl \in E(G),~il \notin E(G)\big{\}} \cup \big{\{} \{i,l\}_2: ~jl \in E(G),~il \in E(G) \big{\}}.$$
Now assumme that $i$ and $j$ are two adjacent vertices in $G$. The expanded forms of the (contracted) vertex $\{i,j\}_2$ of $S[G,2]$ are $ij$ and $ji$. Therefore,
\begin{eqnarray*}
N_{S[G,2]} \big{(}\{i, j\}_2\big{)} = && \!\!\!\!\!\! \big{\{} il: ~jl \in E(G),~il \notin E(G)\big{\}} \cup \big{\{} \{i,l\}_2: ~jl \in E(G),~il \in E(G)\big{\}} \\
&\cup &\!\!\! \big{\{} jl: ~il \in E(G),~jl \notin E(G)\big{\}} \cup \big{\{} \{j,l\}_2: ~il \in E(G),~jl \in E(G)\big{\}}.
\end{eqnarray*}
\begin{lemma} \label{pro 1} \label{AdjacentStep2}
	Two contracted vertices
	$\{i,j\}_2$ and $\{l,k\}_2$
	are adjacent in $S[G,2]$, if and only if
	$|\{i,j\}\cap \{l,k\}|=1$
	and the induced subgraph of $G$ on $3$-set
	$\{i,j,k,l\}$
	 is a triangle.
\end{lemma}
\begin{proof}
	First, suppose that two (newly) contracted vertices
	$\{i,j\}_2$ and $\{l,k\}_2$
	are adjacent in  
	$S[G,2]$. 
	Thus, $i$ and $j$ are two adjacent vertices in $G$, and $l,k$ are adjacent in $G$ as well.
	According to the construction method, two contracted vertices are adjacent just when both are in the same copy. Hence, $|\{i,j\}\cap \{l,k\}|=1$.
	Without loss of generality, assume that $l=i$ i.e. these two adjacent (contracted) vertices are in the $i-th$ copy $S_i[G,2]$. 
Thus $i$ is adjacent to both of vertices $j$ and $k$
	in $G$.
	Two expanded forms of
	$\{i,j\}_2$
	are
	$ij$ and $ji$
	that are connected with an edge in $S(G, 2)$.
	Similarly, two expanded forms of
	$\{i,k\}_2$ are $ik$ and $ki$.
	By considering the structure of $S(G,t)$ and since two vertices $ij$ and $ki$ are not adjacent nor $ji$ with $ki$, we must have $(ij)(ik) \in E(S(G,t))$.
	Thus $jk \in E(G)$ and hence, three vertices $i,j,k$ form a triangle in $G$.
Now the converse is also obvious.
\end{proof}
Next, consider the structure of $S[G,t]$ with $t\geq 3$.
\begin{lemma}\label{Lem1}
Let $t\geq3$ be an integer.
If  two contracted vertices are adjacent in
$S[G,t]$,
then at least one of them is obtained by contracting a linking edge at step $2$.	 
\end{lemma}
\begin{proof}
Let 
${\bf u}=u_1 u_2 \dots u_r\{i,j\}_{t-r}$
be a vertex in $S[G,t]$ which is contracted at step $(t-r)\geq 3$, $ij\in E(G)$.
The expanded forms of this vertex are
${\bf z}=u_1u_2 \dots u_r ijj \dots j$
and
${\bf z^{'}}=u_1u_2 \dots u_rjii \dots i$.
By considering the adjacency rule in $S(G,t)$, since $t-r\geq 3$ the vertex  $\bf{z}$ is adjacent just to a vertex in the form $u_1u_2 \dots u_r ijj \dots jl$  with $l \! \in N_G(j)$. 
According to the construction process, this vertex is contracted at step $2$ and becomes in the form
$u_1u_2 \dots u_r ijj \dots j\{j,l\}_2$.  
Similarly,  ${\bf z^{'}}$ in $S(G,t)$
is adjacent just to a vertex in the form
$u_1u_2 \dots u_r jii \dots il^{'}$
with
$l^{'} \!\!\in N_G(i)$.
Hence,  in $S[G,t]$ the contracted vertex ${\bf u}={\bf z^{'}}$
is adjacent just to a vertex in the form
$u_1u_2 \dots u_r jii \dots i\{i,l^{'}\}_2$.
Thus each neighbour of ${\bf u}$ is in the form
$u_1u_2 \dots u_r ijj \dots j\{j,l\}_2$
or $u_1u_2 \dots u_r jii \dots i\{i,l^{'}\}_2$ which  are contracted vertices at step two.
\end{proof}
\begin{corollary}  \label{NewContIndep}
	If $t\geq 3$, then the set $\big\{ \{i,j\}_t:ij\in E(G) \big\}$ consisting of all (newly)  contracted vertices at step $t$, forms an independant set in $S[G,t]$.
\end{corollary}
Now by considering  the adjacency rule in $S(G,t)$, the structure of $S[G,t]$ which is obtained by contracting the linking edges of $S(G,t)$ and Lemma \ref{Lem1}, we get the following result.
\begin{proposition} \label{Neighbors}
Assume that $t\geq 3$. For the neighbors of typical vertices ${\bf{z}}=\{i,j\}_t$, ${\bf{y}}=y_1y_2 \dots y_t$ and ${\bf{x}}=x_1x_2 \dots x_r \{i,j\}_{t-r}$ in $S[G,t]$,  where $ij \in E(G)$ and $1 \leqslant r \leqslant t-2$, we have the following statements.\\ 
{\bf i)} $N_{S[G,t]}({\bf{z}})$ is equal to
\[
\big{\{} ijj \dots j\{j,l\}_2:~ lj \in E(G) \big{\}} ~\!\! \cup ~\!\! \big{\{} jii \dots i\{i,l\}_2:~ li \in E(G) \big{\}}.
\]
{\bf ii)} $N_{S[G,t]}({\bf{y}})$ is equal to
\[
\big{\{} y_1y_2  \dots y_{t-1}l: ly_t \in E(G),  ly_{t-1} \notin E(G) \big{\}} ~\!\! \cup ~\!\! \big{\{} y_1y_2  \dots y_{t-2} \{y_{t-1}, l\}_2: ly_t, ~ ly_{t-1} \in E(G) \big{\}}.
\]
{\bf iii)} If $(t-r)\geqslant 3$, then $N_{S[G,t]}({\bf{x}})$ is given by
\[
\big{\{} x_1x_2  \dots x_{r}ijj \dots j\{j,l\}_2 : lj \in E(G) \big{\}} ~\!\! \cup~\!\! 
 \big{\{} x_1x_2  \dots x_{r} jii \dots i\{i,l\}_2 : li \in E(G) \big{\}}.
 \]
{\bf iv)} If $(t-r)=2$, then $N_{S[G,t]}({\bf{x}})$ is equal to
\begin{eqnarray*}
\big{\{} x_1x_2  \dots x_{t-2}il: il \notin E(G), jl \in E(G) \big{\}} ~\!\!\cup~\!\!
\big{\{} x_1x_2  \dots x_{t-2} \{i, l\}_2: il \in E(G), jl \in E(G) \big{\}}~\!\!\cup  \\
\big{\{} x_1x_2  \dots x_{t-2}jl: jl \notin E(G), il \in E(G) \big{\}} ~\!\!\cup~\!\!
\big{\{} x_1x_2  \dots x_{t-2} \{j, l\}_2: jl \in E(G), il \in E(G) \big{\}}
 ~\!\!\cup~\! {\bf \Omega}
 \end{eqnarray*}
 in which
 \begin{eqnarray*}
 	{\bf \Omega}=
 	\begin{cases}
 		\emptyset~~~~~~~~~~~~~~~~~~~~~~\! x_1x_2  \dots x_{t-2} \notin \{ijj \dots j, jii \dots i\} ,\\
 		\{\{i,j\}_t\} ~~~~~~~~~~~~~~\!\! x_1x_2  \dots x_{t-2} \in \{ijj \dots j, jii \dots i\}.\\
 	\end{cases}\\
 \end{eqnarray*}
\end{proposition}
Obviously, when $G$ contains a triabngle then $S[G,t]$ (which has at least one subgraph isomorphic to $G$) contains a triangle. In the following result we show that the converse is also true and hence, $G$ contains a triabngle if and only if $S[G,t]$ does.
\begin{proposition} \label{TriangleFree}
	If  $G$  is a triangle-free graph, then $S[G,t]$ is triangle-free.
\end{proposition}
\begin{proof}
Let $G$ be a triangle-free graph and assume on the contrary  that $S[G,t]$ contains a triangle with three vertices $A, B,C$.
If  $A=a_1a_2 \dots a_t$, 	$B=b_1b_2 \dots b_t$ and $C=c_1c_2 \dots c_t$  are three non-contracted vertices in $S[G,t]$,  then we must have $a_{t-1}a_t,b_{t-1}b_t,c_{t-1}c_t \notin E(G)$. Since $A, B,C$ are pairwise adjacent, we must have $a_tb_t,b_tc_t,c_ta_t \in E(G)$ which means that three (distinct) vertice $a_t,b_t,c_t$ form a triangle in $G$, a contradiction.
Thus assume that at least one of these three vertices, say $A$, is a contracted vertex (at some step).
If $A=a_1a_2\cdots a_r \{i,j\}_{t-r}$ is a contracted vertex at step $(t-r)\geq 3$, then Lemma \ref{Lem1} implies that its neighbors $B$ and $C$ are contracted vertices at step two and hence may be considered as vertices in some copies of $S[G,2]$. Since $B$ and $C$ are adjacent, Lemma \ref{AdjacentStep2} implies that $G$ has a triangle, a contradiction.
By similar arguments, we can assume that just one vertex in $\{A,B,C\}$ is a contracted vertex at step $2$ and two others are non-contracted vertices. Thus, assume that $A=a_1a_2\cdots a_{t-2}\{i,j\}_2$ in which $ij\in E(G)$.
We may have $t=2$ or $t\geq 3$. By considering the argument before 
Lemma \ref{AdjacentStep2}, the case (iv) in Proposition \ref{Neighbors}, the adjacency of $B$ and $C$, and without loss of generality, we can assume that $B=a_1a_2\cdots a_{t-2}il$ and $a_1a_2\cdots a_{t-2}il'$ in which $lj,l'j,ll'\in E(G)$. Then, $j,l,l'$ form a triangle in $G$, a contradiction. Thus, $S[G,t]$ is a triangle-free graph, as required.
\end{proof}
Proposition \ref{TriangleFree} means that the existance of a clique of size 3 in $S[G,t]$ leads to the existance of a clique with the similar size 3 in $G$. In the following result we show that each clique of size at least four in $S[G,t]$ is in fact completely inside a copy $S_i[G,t]$  for some $i$.

\begin{proposition} \label{pro8} \label{Clique4First}
	Let $C$ be a clique in $S[G,t]$ with $|C|\geqslant 4$. Then for some $1\leq i\leq n$ we have $C \subseteq V(S_i[G,t])$.
\end{proposition}
\begin{proof}
Let ${\bf z} \in C$ and hence ${\bf z} \in V(S_i[G,t])$ for some $i\in \{1,2,...,n\}$, $n=|V(G)|$.
If ${\bf z}$ do not be a (newly) contracted vertex at step $t$, then three cases (ii) to (iv) in Proposition \ref{Neighbors} imply that each neighbour of ${\bf z}$ is in $V(S_i[G,t])$ and hence, $C \subseteq V(S_i[G,t])$ which completes the proof.
Thus, assume that each vertex of $C$ is a (newly) contracted vertex at step $t$. Specially, we have ${\bf z}=\{i,j\}_t$ for some $j\in \{1,2,...,n\}$ in which $ij\in E(G)$.
By the case (i) in Proposition \ref{Neighbors}, each neighbour of ${\bf z}$ is in $V(S_i[G,t]) \cup V(S_j[G,t])$ and this implies that $C \subseteq V(S_i[G,t]) \cup V(S_j[G,t])$.
If $C \subseteq V(S_i[G,t])$ or $C \subseteq V(S_j[G,t])$, then the proof is complete. 
Therefore, assume (on the contrary) that there exist two vertices ${\bf u} \in C \cap V(S_i[G,t])$ and ${\bf  v} \in C \cap V(S_j[G,t])$ with ${\bf  uz} \in E(S[G,t])$ and ${\bf vz} \in E(S[G,t])$. Since $C$ is a clique, we have ${\bf uv} \in E(S[G,t])$ and hence, there exists $k\in \{1,2,...,n\}\setminus \{i,j\}$ such that ${\bf u}=\{i,k\}_t$ and ${\bf v}=\{j,k\}_t$ with $\{ik,jk\}\subseteq E(G)$.
Thus, ${\bf u}$ has no neighbour in $V(S_j[G,t])$ except ${\bf z}$ and ${\bf v}$, and ${\bf v}$ has no neighbour in $V(S_i[G,t])$ except ${\bf z}$ and ${\bf u}$. Thus $C=\{{\bf z},{\bf u},{\bf v}\}$ and $|C|\neq 4$, a contradiction which completes the  proof.
\end{proof}
Assume that there exists a clique $C$ with at least $4$ vertices  in $S[G,t]$. Then, by Proposition \ref{Clique4First} there exists $1\leq i_1 \leq n$ such that $C \subseteq V\big(S_{i_1}[G,t]\big)$. Since $S_{i_1}[G,t]$ is isomorphic to $S[G,t-1]$, for convenient we can consider  $S_{i_1}[G,t]$ as $S[G,t-1]$ and hence,  $C$ is in fact a clique with at least $4$ vertices  in $S[G,t-1]$. Similarly, by Proposition \ref{Clique4First} there exists $1\leq i_2 \leq n$ such that $C \subseteq V\big(S_{i_2}[G,t-1]\big)$. By repeatig this process, we can find a sequence of numbers $i_1,i_2,...,i_{t-1}$ such that
$$C \subseteq V\big(S_{i_{t-1}}[G,2]\big)\subseteq \dots \subseteq V\big(S_{i_{2}}[G,t-1]\big) \subseteq V\big(S_{i_{1}}[G,t]\big).$$
Since $S_{i_{t-1}}[G,2]$ can be considered as $G$, we obtain the following result which will be applied in the next theorem for determining  the exact value of the clique number of $S[G,t]$.
\begin{corollary} \label{Clique4Second}
	If $C$ is a clique with at least $4$ vertices in $S[G,t]$, then $C$ is completely inside a copy of $G$. 
\end{corollary}
\begin{theorem} \label{CliqueNumb}
	For each graph $G$ and each positive integer $t$, we have 	$\omega\big(S[G,t]\big)=\omega(G)$.
\end{theorem}
\begin{proof}
Since $G$ is a subgraph of $S[G,t]$, we have 	$\omega(G)\leq \omega(S[G,t])$. 
If 	$\omega(G)=1$, then $G$  does not have any edge and  hence $S[G,t]$ does not have any edge, which implies that 
$\omega (S[G,t])=1=\omega (G)$, as required.
If $\omega (G)=2$, then $G$ is a triangle-free graph and 	Proposition \ref{TriangleFree} implies that $S[G,t]$ is also triangle-free. Hence, $\omega (S[G,t])=2=\omega (G)$, as required.
If   $\omega(G)= 3$, then Corollary \ref{Clique4Second} implies that $\omega (S[G,t]) \leq 3$ and hence, $\omega (S[G,t])=3=\omega (G)$.
Thus assume that $\omega(G)\geq 4$. Again, Corollary \ref{Clique4Second} implies that $\omega (S[G,t]) \leq \omega (G)$. Hence,
$\omega(S[G,t]) = \omega(G)$, which completes the proof.
\end{proof}
The clique number of each graph is a lower bound for its chromatic number. Thus by using Theorem \ref{CliqueNumb} we see that $\omega(G)\leq \omega(S[G,t]) \leq \chi(S[G,t])$. Since $G$ is a subgraph of $S[G,t]$ we can improve this lower bound as $\chi(G)\leq \chi(S[G,t])$.  In 2014  Jakovac \cite{H} proved that $\chi (S[K_n,t])=\chi(K_n)$. Thus, the lower bound $\chi(G)\leq \chi(S[G,t])$ is attainable by completes graphs. In Theorem \ref{thm20} we will show that this lower bound is attainable for bipartite graphs. Also, in the following we will prove that $\chi(S[G,t]) \leq 1+\chi(G)$ which provides an interesting upper bound.
It is well known that the chromatic number of a graph is 2 if and only if it is a bipartite graph (with non-empty edge set). Also, each tree is a bipartite graph. 
In \cite{Y} it is proved that if $T$ is a tree, then the generalized Sierpi\'nski $S(T,t)$ is a tree. Poroposition \ref{ob1} and Observation \ref{ob2} implies that when $T$ is a tree then the generalized Sierpi\'nski gasket $S[T,t]$ is also a tree and the following result directly follows.
\begin{proposition} \label{TreeChrom}
	For each tree $T$ of order $n\geq 2$ and each integer $t\geq1$, 
	$S[T,t]$ is a tree and $\chi (S[T,t])=2$.
\end{proposition}
\begin{theorem} \label{TH 12}
	For each graph $G$, we have $\chi (S[G,2])=\chi (G)$.
\end{theorem}
\begin{proof}
Suppose that 	$\chi (G)=k$ 	and let 	$f: V(G)\rightarrow \{1,2, \dots k\}$ be a proper vertex-coloring  of $G$.
	Since $G$ is a subgraph of $S[G,2]$, we have $k=\chi (G)\leq \chi (S[G,2])$.
Now define the (coloring) function  $g: V(S[G,2])\rightarrow \{1,2, \dots k\}$ as 
$g(xy)=f(x)+f(y)$ 	and $g(\{x^{\prime},y^{\prime}\}_2)=f(x^{\prime})+f(y^{\prime})$, in which the color numbers are considered in module $k$. 	We want to show that $g$ is a proper vertex-coloring of $S[G,2]$ and hence,  $\chi (S[G,2]) \leq k$, which  completes the proof.
	Let ${\bf u}$ and ${\bf v}$ be two adjacent vertices in 	$S[G,2]$.
Thus one of the following three cases will occur.\\
{\bf Case i. } 	${\bf u}= xy$ and ${\bf v}= xy_1$ for some $x, y, y_1 \in V(G)$ with $yy_1\in E(G)$. \\
	In this case, two vertices ${\bf u}$ and ${\bf v}$ are in fact inside the  copy  $S_x[G,2]$ which is isomorphic to $G$. 
Since $yy_1 \in E(G)$ and $f$ is  the proper vertex-coloring of $G$, we have 	$f(y)\neq f(y_1)$ 	and hence,
$$g(xy)= f(x)+f(y)\neq  f(x)+f(y_1)=g(xy_1).$$
In fact, $g$ provides a proper $k$-coloring for $S_x[G,2]$ by permuting the colors assigned by $f$ to $V(G)$. \\
{\bf Case ii.} 	${\bf u}=\{x,y\}_2$ and ${\bf v}=xy_1$ for some $x, y, y_1 \in V(G)$ with $yy_1 \in E(G)$. \\
Since $f(y)\neq f(y_1)$, by the defintion of $g$ we obtain
$$g(\{x,y\}_2) = f(x)+f(y)\neq f(x)+f(y_1)= g(xy_1).$$
{\bf Case iii. } 	${\bf u}=\{x,y\}_2$ and ${\bf v}= \{x,y_1\}_2$ for some $x, y, y_1 \in V(G)$ with $yy_1 \in E(G)$. \\
Similarly, from $f(y)\neq f(y_1)$ we obtain $g(\{x,y\}_2) \neq  g(\{x,y_1\}_2)$. \\
Hence, in each case the colors assigned by $g$ to two adjacent vertices ${\bf u}$ and ${\bf v}$ are different, and the proof is complete.
\end{proof}
\begin{theorem} \label{UpBound}
	Let $G$ be an arbitrary graph and $t\geq 1$ be a positive integer. Then we have
	$\chi (S[G,t])\leq \chi (G)+1$.
\end{theorem}
\begin{proof}
Assume that $|V(G)|=n$ and $\chi (G)=k$. 
	For $t=1$ we have $S[G,1]=G$ and  $\chi (S[G,1])= \chi (G)$.
	For $t=2$, Theorem \ref{TH 12} implies that 
	$\chi (S[G,2])=k$ and hence, there exists a proper $k$ coloring
	$F_2: V(S[G,2])\rightarrow \{1,2, \dots ,k\}$.
Now suppose that $t=3$. 
Since for each $1\leq i\leq n$ 	 the $i-th$ copy $S_i[G,3]$  in $S[G,3]$ is isomorphic to $S[G,2]$,  
there exists	a  proper  $k$-coloring $F_{2,i} : V(S_i[G,3])\rightarrow \{1,2, \dots ,k\}$ of the vertices of $S_i[G,3]$. 
Let $A_3 = \big{\{ }\{i,j\}_3 : ij \in E(G) \big{\}}$ be the set of all (newly) contracted vertices at step $3$ in $S[G,3]$. 
By Corollary \ref{NewContIndep}, $A_3$ is an independent set in $S[G,3]$.
Define the (coloring) function 	$F_3: V(S[G,3])\rightarrow \{1, 2, \dots ,k, k+1\}$ as 
\begin{eqnarray*}
F_3({\bf u})=
		\begin{cases}
		k+1 ~~~~~~~~~~~~~~~\! {\bf u} \in A_3,\\
		F_{2,i}({\bf u}) ~~~~~~~~~~~~~\!
		{\bf u} \in V\big{(}S_i[G,2]\big{)} \backslash A_3,~ 1\leqslant i \leqslant n 		.\\
		\end{cases}\\
	\end{eqnarray*}
Since $A_3$ is an independent set, each $F_{2,i}$ is a proper $k$-coloring for the copy $S_i[G,3]$ and each pair of distinct copies $S_i[G,3]$ and $S_j[G,3]$ share at most one vertex $\{i,j\}_3$ in $A_3$ when $ij\in E(G)$, $F_3$ is a proper $(k+1)$-coloring for $S[G,3]$ in which the color class of all vertices with color $k+1$ is $A_3$.
 Specially, $\chi (S[G,3]) \leqslant k+1$.
 Now let $t\geqslant 4$ and in an inductive way, assume that $\chi (S[G,t-1]) \leqslant k+1$
	and there exists a proper $(k+1)$-coloring $F_{t-1}$ of $S[G,t-1]$ in which only contracted vertices at some step $s \in \{3, 4, \dots , t-1\}$ 
	received the color ``$k+1$".
Note that $S_i[G,t]$ is isomorphic to $S[G,t-1]$ for each  $1\leq i\leq n$. Hence, for each  $1\leq i\leq n$ let $F_{t-1, i}$ be a proper $(k+1)$-coloring of $S_i[G,t]$  in which only contracted vertices at some step $3\leqslant s \leqslant t-1$ 	received the color ``$k+1$". 
Let $A_t$ be the (independent) set of all (newly) contracted vertices at step $t$ in $S[G,t]$
and define the (coloring) function $F_t: V(S[G,t])\rightarrow \{1, 2, \dots, k+1\}$ as 
\begin{eqnarray*}
F_t({\bf u})=
		\begin{cases}
			k+1 ~~~~~~~~~~~~~~~~~\! {\bf u} \in A_t,\\
			F_{t-1,i}({\bf u}) ~~~~~~~~~~~~\!
			{\bf u} \in V(S_i[G,t]) \backslash A_t,~ i \in \{1, 2, \dots, n\}
			.\\
		\end{cases}\\
	\end{eqnarray*}
Similarly,  $F_t$ is a proper coloring for $S[G,t]$ and hence	$\chi (S[G,t])\leq k+1$, which completes the proof.
\end{proof}
Since $G$ is a subgraph of $S[G,t]$, we have $\chi(G)\leq \chi(S[G,t])$ and hence, the following result directly follows from Theorem \ref{UpBound} which indicates that $\chi(S[G,t])$ is equal to eithter $\chi(G)$ or $\chi(G)+1$.
\begin{corollary} \label{UpLowBound}
	For each grapgh $G$ and each positive integer $t$ we have
	$$\chi (G) \leq\chi \big(S[G,t]\big)\leq \chi (G)+1.$$
\end{corollary}
In the following theorem, we generalize Proposition \ref{TreeChrom} for bipartite graphs.
\begin{theorem}\label{thm20}
	If $\chi (G)=2$, then for each positive integer $t$ we have $\chi \left(S[G,t]\right)=2$.
\end{theorem}
\begin{proof}
By Corollary \ref{UpLowBound}, it is sufficient to provide a proper $2$-coloring for $S[G,t]$. Assume that  	$f: V(G)\rightarrow \{1,2\}$ is a proper $2$-coloring for $G$. The case $t=1$ is obvious since $S[G,1]=G$.
For $t=2$,  Theorem \ref{TH 12} shows that $\chi \left(S[G,2]\right)=\chi(G)=2$ and its proof provides  the proper $2$-coloring function
$g:V(S[G,2])\rightarrow \{1,2\}$ defined as 
$g(u_1u_2)= f(u_1)+f(u_2)$ and $g(\{u_1,u_2\}_2)= f(u_1)+f(u_2)$ (where the summations are considered in module $2$).
Note that for each contracted vertex $\{u_1, u_2\}_2$ in $V\big(S[G,2]\big)$ we have $u_1u_2\in E(G)$ and hence, $f(u_1)\neq f(u_2)$. 
Thus, $f(u_1)+f(u_2)$ is eithter $1+2$ or $2+1$, which in module $2$ becomes $1$. 
Hence, for each contracted vertex $\{u_1, u_2\}_2$ we have $g\left(\{u_1, u_2\}_2\right)=1$.
Also, for each extreem vertex like $uu$ we have $g(uu)=2$.
Now let $t=3$ and consider the graph $S[G,3]$.
Since for each $1\leq i \leq |V(G)|$, 	$S_i[G,3]$ is isomorphic to $S[G,2]$, we can use the function $g$ to provide a proper $2$-coloring  for  the vertices  of each $S_i[G,3]$  in such a way that all contracted vertices at step two (like $i\{u_1,u_2\}_2$) receive the color $1$ and the  color of each vertex in the form $iuu$ becomes $2$.
In $S[G,3]$, the expanded forms of each (newly) contarcted vertex $\{i,j\}_3$ are $ijj$ and $jii$ which both of them received the color $2$ in colorings of $S_i[G,3]$ and $S_j[G,3]$, i.e. they agree in color, and $\{i,j\}_3$ can be considered as a vertex in $S[G,3]$ whose color is $2$ and its neighbours (which are contracted vertices at step two) have color $1$. This provides 
 a proper $2$-coloring $S[G,3]$ in which all contracted vertices at step two have color $1$.
By an inductive way, assume that $t\geq 4$ and there exists a proper $2$-coloring of $S[G,t-1]$ in which all contracted vertices at step two (like  $u_1u_2...u_{t-3}\{i,j\}_2$) have color $1$.
By using the same coloring on each copy $S_i[G,t]$ and by assigning  color $2$ to each newly contracted vertex at step $t$, Corollary \ref{NewContIndep} implies that a proper $2$-coloring for $S[G,t]$ is obtained and the proof is complete.
\end{proof}
Since the chromatic  number of a graph $G$ is $2$ if and only if $G$ is a bipartite graph with a non-empty edge set, 
the following result directly follows from Theorem \ref{thm20}.
\begin{corollary}
	If $G$ is a bipartite graph, then $S[G,t]$ is bipartite.
\end{corollary}
By considering these results, we provide the following conjecture.

\begin{con}
	For each graph $G$ and each integer $t\geq1$ we have $\chi \left(S[G,t]\right)=\chi (G)$.
\end{con}
{\bf Acknowledgments:} The authors warmly thank the anonymous referees for reading this manuscript
very carefully and providing numerous valuable corrections and suggestions which improve the quality of this paper.
\\\\
The authors declare that they have no competing interests.


\begin{thebibliography}{10}
	
\bibitem{TOC}
F. Attarzadeh, A. Abbasi and M. Gholamnia Taleshani, 
Some properties of the generalized sierpi\'{n}ski gasket graphs, 2024, DOI: 10.22108/TOC.2024.138919.2098.


\bibitem{D}
G. Della Vecchia and C. Sanges, A recursively scalable network VLSI implementation, Future Gener. Comput. Syst. Sci. Eng. 4 (1988), 235-243. 


\bibitem{Garey}
M. R. Garey and D. S. Johnson: Computers and Intractability: A Guide to the Theory of NP-Completeness. W. H. Freeman \& Co. 1979.

\bibitem{P}
S. Gravier, M. Kovse and A. Parreau, Generalized Sierpi\'{n}ski graphs, in Posters at EuroComb’11, Renyi Institute, Budapest, (2011). 
\bibitem{Q}
S. Gravier, M. Kov$\check{s}$e, M. Mollard, J. Moncel and A. Parreau, New results on variants of covering codes in Sierpi\'{n}ski graphs, 	Des. Codes Cryptogr. 69 (2) (2013), 181-188. 
\bibitem{R}
A.M. Hinz, S. Klav$\check{z}$ar and S.S. Zemljic, A survey and classification of Sierpi\'{n}ski-type graphs, Discret. Appl. Math. 217 (2017), 565-600. 

\bibitem{H}
M. Jakovac, A 2-parametric generalization of Sierpinski gasket graphs, Ars Combin., 116 (2014) 395-405. 
\bibitem{ColorSerpLike}
M. Jakovac, S. Klav$\check{z}$ar, Vertex-, edge-, and total-colorings of Sierpi\'nski-like graphs, Discrete Mathematics,  309 (2009) 1548-1556.


\bibitem{FFF}
Khatibi, M.  Behtoei A. and Attarzadeh, F.  Degree sequence of the generalized Sierpi\'{n}ski graph, Contrib. Discrete Math., 3 (1715-0868) (2020), 88-97.
\bibitem{AAA}
S. Klav$\check{z}$ar and  S. Sabrina Zemlji$\check{c}$,
Connectivity and some other properties of generalized Sierpiński graphs, Appl. Anal. Discrete Math., 12 (2018), 401-412.
\bibitem{L}
S. Klav$\check{z}$ar, Coloring Sierpinski graphs and Sierpinski gasket graphs. Taiwan. J. Math., 12 (2008), 513-522. 
\bibitem{B}
S. Klav$\check{z}$ar and U. Milutinovi\'c, Graphs S(n,k) and a variant of the Tower of Hanoi problem, Czechoslovak Math. J., 122 (1997), 95-104. 
\bibitem{T}
S. Klav$\check{z}$ar, U. Milutinovi\'c and C. Petr, 1-perfect codes in Sierpi\'{n}ski graphs, Bull. Aust. Math. Soc., 66 (3) (2002), 369-384. 
\bibitem{V}
D. Parisse, On some metric properties of the Sierpi\'{n}ski graphs $S(n,k)$, Ars Combinatoria 90 (2009), 145-160. 

\bibitem{PolymerNet}
J. A. Rodriguez-Vela\'{z}quez, J. Tom\'{a}s–Andreu, On the Randi\'{c} index of
polymeric networks modelled by generalized Sierpi\'{n}ski graphs, MATCH Commun. Math. Comput. Chem. 74 (2015) 145-160. 

\bibitem{Y}
J. A. Rodriguez-Velazquez, E. D. Rodriguez-Bazan, A. Estrada-Moreno, On generalized Sierpi\'{n}ski graphs, Discuss. Math. Graph T., 37 (3) (2017), 547-560.
\bibitem{F}
R.S. Scorer, P.M. Grundy and C.A.B. Smith, Some binary games, Math. Gaz., 28 (1944), 96-103. 


\end{thebibliography}
\end{document}